\documentclass[11pt]{amsart}
\usepackage[margin=1.2in]{geometry}
\usepackage{amsmath}
\usepackage{amsthm}
\usepackage{hyperref}
\usepackage{soul}
\usepackage[T1]{fontenc}
\usepackage{amsfonts}
\usepackage{todonotes}
\usepackage{color}
\usepackage{enumerate}
\usepackage{amssymb}
\usepackage{amscd}%%%added by Jan for V. Kac 4/23/07
\numberwithin{equation}{section}
\newtheorem{theorem}{Theorem}[section]

\newtheorem{remark}[theorem]{Remark}
\newtheorem{proposition}[theorem]{Proposition}

\theoremstyle{definition}
\newtheorem{definition}{Definition}[section]

\newcommand\E{E}
\newcommand\D{D}
\newcommand\C{B}

\newcommand{\bea}{\begin{eqnarray}}
\newcommand{\eea}{\end{eqnarray}}

\renewcommand{\lim}{\varinjlim}

\newcommand{\Z}{{\mathbb{Z}}}
\newcommand{\R}{{\mathbb{R}}}
\newcommand{\Q}{{\mathbb{Q}}}

\newcommand{\A}{{\mathcal{A}}}

\newcommand{\emme}{{\mathcal M}}
\newcommand{\elle}{{\mathcal L}}

\begin{document}
\title[The homotopy type of the complement  of a toric arrangement]{A  differential algebra and the  homotopy type of the complement  of a toric arrangement}
%\theshorttitle {The homotopy type of the complement  of a toric arrangement}
\author{Corrado De Concini, Giovanni Gaiffi}
\maketitle
\begin{abstract}
We show that the rational homotopy type of the complement of a toric arrangement  is completely determined by two sets of combinatorial data.  This is obtained by introducing  a  differential graded  algebra over \(\Q\)    whose minimal model is equivalent to the Sullivan minimal model of \(\A\). 
\end{abstract}

\section{Introduction}
Let $T\simeq G_m^n$ be a complex $n$ dimensional algebraic torus  and let us denote by $X^*(T)\simeq \mathbb Z^n$ its character group.

A {\em layer} in \(T\) is the subvariety 
$$\mathcal K_{\Gamma,\phi}=\{t\in T|\, \chi(t)=\phi(\chi),\, \forall \chi\in \Gamma\}$$
where $\Gamma$ is a split direct summand of  $X^*(T)$ and  $\phi:\Gamma \to \mathbb C^*$ is a homomorphism.

A  toric arrangement \(\A\) is given by  a  finite set of  layers \(\A=\{\mathcal K_{1},...,\mathcal K_{m}\}\) in $T$; if for every \(i=1,...,m\) the layer \(\mathcal K_i\)  has codimension 1
the arrangement  \(\A\) is called  {\em divisorial}.

We will denote by \(\emme(\A)\) the complement \(T-\bigcup_i \mathcal K_i\) of the arrangement.  
We notice that if we consider the saturation $\tilde {\mathcal A}$ of \(\A\) , i.e. the arrangement consisting of all the  layers which are obtained as connected components of intersections of layers in $\mathcal A$, we have \(\emme(\A)=\emme(\tilde{\mathcal A}) \).

The purpose of this note is to show that the rational homotopy type of \(\emme(\A)\) is completely determined by
 
\begin{enumerate} \item The partially ordered set $\tilde {\mathcal A}$  ordered by reverse inclusion.
\item The set  of lattices $\Gamma\subset X^*(T)$ for $\mathcal K_{\Gamma,\phi}\in \tilde {\mathcal A}$.
\end{enumerate} 
We will call these data {\em the combinatorial data} of \(\A\).

This is obtained by introducing  an object that may be of independent interest: a differential graded  algebra over \(\Q\), defined using the combinatorial data of \(\A\),  whose minimal model is equivalent to the Sullivan minimal model of \(\A\).  In particular we have that  its cohomology is isomorphic to the rational cohomology of \(\emme(\A)\).

Before giving a sketch of our construction, we  recall some  previous results on this subject.
%\subsection{The cohomology of the complement and wonderful models} 

As far as we are aware,  the results  regarding the (rational) homotopy  of \(\emme(\A)\) have been obtained in the divisorial case. In this case, in \cite{DCPtoric} De Concini and Procesi determined the generators of the rational cohomology modules  of \(\emme(\A)\), as well as the ring structure in the case of totally unimodular arrangements. By a rather general approach,  Dupont in \cite{Dup} proved the rational formality of \(\emme(\A)\). In turn, in \cite{CDDMP}, it was shown extending the results in \cite{callegarodelucchi, callegaro2019erratum} and  \cite{Pagariatwo} that 
the data needed in order to state the presentation of the rational  cohomology ring of \(\emme(\A)\)  is fully encoded in the partially ordered set $\tilde {\mathcal A}$. It follows that the combinatorics of the poset $\tilde {\mathcal A}$ determines the rational homotopy  of \(\emme(\A)\).

% rational (and even integral) cohomology ring 
%Again in the divisorial case, Callegaro and  Delucchi computed in \cite{callegarodelucchi, callegaro2019erratum} the graded cohomology ring with integer coefficients, while  the cohomology ring itself was computed by Callegaro, D'Adderio, Delucchi, Migliorini and Pagaria in \cite{CDDMP}. As it was shown by Pagaria in \cite{Pagariatwo}, in the divisorial case  the rational  cohomology algebra of the complement of a toric arrangement is  determined by the poset of layers while  the integer   cohomology algebra is not.

%\textcolor{blue} {Qui mi pare un immane bordello. L'algebra di coomologia non determina il tipo di omotopia razionale a meno che lo spazio non sia formale. Dunque o tu dai una referenza in cui si dimostra questo ma se dai una serie di informazioni sulla coomologia questo non si lega con la frase che segue. Che vuoi dire. Mi fermo e aspetto una risposta. }

Our approach to the study of the rational homotopy type and of the cohomology of \(\emme(\A)\)  in full generality involves the construction of projective wonderful models for \(\emme(\A)\).

Results analogous to those in this note were previously obtained in ~\cite{DCP1} for arrangements of linear subspaces in a projective space using the construction of wonderful models for subspace arrangements and the fundamental results of Morgan in \cite{Mor}. In particular Morgan introduced, in the case of a compactification $V$ of a complex algebraic variety X such that $V\setminus X=D$ is a divisor with normal crossing, a differential graded algebra (which we are going to call a Morgan differential algebra)  
whose minimal model is equivalent to the Sullivan minimal model of \(X\).

%Since then, real and complex wonderful  models of subspace arrangements 
% were investigated from  several  points of view: the cohomology ring in the real case was studied in \cite{etihenkamrai};  some  relevant combinatorial properties  and their relation with discrete geometry and tropical geometry were pointed out  for instance  in \cite{feichtner},  \cite{gaiffipermutonestoedra}, \cite{feichtnersturmfels} and \cite{denham};   the case of complex reflection groups was dealt with  in \cite{hendersonwreath} from the representation theoretic point of view and in \cite{callegarogaiffilochak} from the homotopical point of view.
%The connections between the geometry of these models and the Chow rings of matroids were pointed out first by Feichtner and Yuzvinsky~\cite{feichtneryuz} and then by Adiprasito, Huh and Katz in~\cite{adiprasitokatzhuh}, where they also played a crucial role in the study of some log-concavity problems.

%A recent branch of research focuses on the extension of the theory of wonderful models to  the case of toric arrangements .
%In~\cite{mociwonderful} Moci introduced  non projective wonderful models for the complement \(\emme(\A)\) of a toric arrangement.
In the toric case the idea is to do similar considerations using  the  \emph{projective} wonderful models of \(\emme(\A)\) we constructed  in   \cite{DCG1}  and the presentation by generators and relations of  the  integer cohomology rings of these models and of the strata in their boundary given in \cite{DCG2}.

Indeed   in \cite{mocipagaria2020} for each of these models these presentations were used to  describe its  Morgan differential algebra which determines the rational homotopy type  of \(\emme(\A)\). 
%In turn this was used to show the combinatorial nature of the rational homotopy type of $\emme (\A)$ in the divisorial case.

However  the projective models described  in  \cite{DCG1} do  not depend only  on the combinatorial data of the toric arrangement \(\A\) but also on some extra choices  in  the construction process and indeed the  differential Morgan algebras one obtains  also depend on these choices.

To overcome this problem we are going to construct a new  differential graded algebra as a direct limit of the   differential Morgan algebras of the projective wonderful models described  in  \cite{DCG1}. This algebra, based on the notion of the  \emph{ring of conditions of $T$},  is rather large and it is not the Morgan algebra of any compactification of $\emme (\A)$. 

 However we show in  Proposition \ref{quasifinale}, that it is quasi isomorphic to any of the Morgan algebras of the projective wonderful models of $\emme(\A)$  and it has a  simple presentation which depends only on the combinatorial data of \(\A\).

Let us describe more in detail the structure of this paper. 
First  (in Section \ref{sec:conditions}) we briefly provide   a self contained presentation of the ring of conditions  \(\mathcal C(T)\) of the torus \(T\). 
We recall that it  was shown by Fulton and Sturmfels in \cite{FS}  that this ring over  \(\Q\)   is isomorphic to the McMullen polytope algebra (see \cite{McMullen}, and \cite{Morelli} for a similar construction) and is the direct limit of the rational Chow rings of all the compactifications of \(T\). For other descriptions of the ring of conditions of the torus the reader can see  for instance  \cite{Br}, \cite{KavKhov} (and \cite{DCPII}  where rings of conditions appeared in a more general setting).  
 
In Section \ref{sec:conditions} we first introduce  an  equivariant version \(B_T(T)\) of the ring of conditions as follows. 
Let us consider the lattice of one parameter subgroups  $X_*(T)=\hom (X^*(T),\mathbb Z)$  and the vector space $V=X_*(T)\otimes_{\mathbb Z} \mathbb R$. We denote by \(\Sigma\) the space of  the continuous functions \(f\) on \(V\) such that  \(f(X_*(T))\subset \Q\)  and there exists a  smooth projective fan such that \(f\) restricted to every face of this fan is linear. Then we consider the    algebra \(B_T(T)\) of continuous functions  \(V\)  generated by \(\Sigma\) and finally we obtain  \(\mathcal C(T)\) as the quotient of \(B_T(T)\) modulo the ideal generated by a basis of \(X^*(T)\).

Then in Section \ref{sec:differentialgraded}  we construct a differential graded algebra \(\mathcal C\) whose degree 0 term is \(\mathcal C(T)\).
This algebra is the direct limit of all the differential graded Morgan algebras associated to the compactifications of \(T\).

So far the toric arrangement \(\A\) has not been taken into account. It appears in Section \ref{sec:toric}, where we first recall from \cite{DCG1} the construction of the projective wonderful model \(Y(X_{\mathcal F})\) associated to \(\tilde \A\) and to a  suitable smooth projective fan \(\mathcal F\). Then we   recall from \cite{DCG2} the presentation of the cohomology of  \(Y(X_{\mathcal F})\) and of its strata in the boundary and we construct, following  Moci and Pagaria  (see  \cite{mocipagaria2020}),  the  Morgan differential algebra \(N_{\mathcal F}\) for 
\(Y(X_{\mathcal F})\). 

Finally, in Section \ref{sec:limit} we introduce the differential  graded algebra \(N\) as a direct limit of the algebras \(N_{\mathcal F}\)
and  we present it by generators and relations (see Theorem \ref{teo:limite})  starting from \(\mathcal C \otimes \mathcal B\), where \(\mathcal C\) is the limit algebra mentioned above and \(\mathcal B\) is a quotient of a Weyl algebra.  
We immediately obtain  Proposition \ref{quasifinale} (the minimal model of \(N\) is isomorphic to the minimal model of \(\emme(\A)\))  and, since the generators and relations of \(N\) depend only on the combinatorial data of  \(\A\), we deduce Theorem \ref{finale} on the rational homotopy type.

\section{The ring of conditions, recollections}
\label{sec:conditions}

Let $T\simeq (C^*)^n$ be a complex $n$ dimensional algebraic torus. Denote by $X^*(T)\simeq \mathbb Z^n$ its character group and by $X_*(T)=\hom (X^*(T),\mathbb Z)$ its lattice of one parameter subgroups. We set $V=X_*(T)\otimes_{\mathbb Z} \mathbb R$. 

We take a rational smooth projective  fan $\mathcal F$  in $V$ and let  $\Gamma_{\mathcal F}=\{c_1,\ldots c_N\}$ be vertices, that is the set of primitive vectors in the one dimensional cones (rays) of $\mathcal F$. Any cone $C\in \mathcal F$ is of the form $$C=C(c_{i_1}, \dots c_{i_k})=\{v=\sum_{r=1}^ka_rc_{i_r}|\ a_r\geq 0\}$$ with $c_{i_1}, \dots c_{i_k}$  the basis of  a split direct summand in $X_*(T)$. 
%\textcolor{red}
%{We set $U:=\mathbb Q^{\Gamma}$,} \textcolor{blue}{mi sembra che non si usi nell'immediato}

Let us take  variables $x_c,$ $c\in  \Gamma_{\mathcal F}$, and in the polynomial ring $\mathbb Q[x_c]_{|c\in  \Gamma_{\mathcal F}}$ take the ideal $I_{\mathcal F}$ generated by the monomials $m_J=\prod_{c\in J}x_c$ for all subsets of $J\subset \Gamma_{\mathcal F}$ for which $C(c_j)_{j\in J}$ is not  a cone in $\mathcal F$.  
Take the algebra $A_\mathcal F:=\mathbb Q[x_{c_1},\ldots x_{c_N}]/I_{\mathcal F}$. $A_\mathcal F$ is the Stanley-Reisner ring of $\mathcal F$ and it   is the equivariant cohomology ring of the toric variety $X_{\mathcal F}$ corresponding to $\mathcal F$ (see \cite{Br} Corollary 1.3 and Proposition 2.2).
The algebra $A_\mathcal F$ inherits a grading from the grading of $\mathbb Q[x_{c_1},\ldots x_{c_N}]$ in which deg $x_c=2$ for all $c\in \Gamma_{\mathcal F}$. The degree 2 part is spanned by the classes of the elements $x_c$ and we denote it  by $S_\mathcal F$.
%\textcolor{red}{nelle righe sopra comincia una ambiguit\`a di notazione sulle variabili $x$. sono infatte denominate sia $x_1,..,x_n$ sia $x_{c_1},...$ ovvero $x_c$ con $c\in \Gamma_{\mathcal F}$}
We may associate to  each $x_c$ the  function $s_c$ on $V$ defined as follows. We set  $s_c(d)=\delta_{c,d}$. Then for $v\in V$ there exist a unique  cone $C=C(c_{i_1}, \dots c_{i_k})\in \mathcal F$ such that $v$ lies in the relative interior of  $C$, so $v=\sum_{r=1}^ka_rc_{i_r}$, $a_r>0$. We then set 
$$s_c(v)=\begin{cases}0\ \ \text {if}\ \ c\neq c_{i_r} \ \forall r\\ a_r\ \ \text {if}\ \ c=c_{ i_r} \end{cases}$$
Notice that $s_c$ is continuous so that sending \(x_c\) to \(s_c\)   we get a homomorphism 
 $$\rho_{\mathcal F}:\mathbb Q[x_{c_1},\ldots x_{c_N}]\to C(V),$$
where $C(V)$ is the algebra of continuous functions on $V$.

The functions $s_c$ are linearly independent in $C(V)$. Their span  will be identified with  $S_\mathcal F$ and denoted by the same letter.

 The following result is well known and we prove it for completeness.
\begin{proposition} \label{conti}\begin{enumerate}
\item The space  $S_{\mathcal F}\subset C(V)$ is the space of continuous functions on $V$ with the property that their restriction to 
 each cone of $\mathcal F$ is linear. 
 \item The ideal $I_{\mathcal F}$ is the kernel of $\rho_{\mathcal F}$. In particular we obtain an inclusion $$\mu_{\mathcal F}:A_{\mathcal F}\rightarrow  C(V).$$
 \item Let $\mathcal G$ be a smooth  refinement of $\mathcal F$, that is every cone in $\mathcal F$ is subdivided by  cones in $\mathcal G$. 
 We know that there is a map
 $$\gamma_{\mathcal G}^{\mathcal F}:A_{\mathcal F}\to A_{\mathcal G}.$$
 Then \begin{equation}\label{compo}\mu_{\mathcal F}=\mu_{\mathcal G}\gamma_{\mathcal G}^{\mathcal F}.\end{equation}
 \end{enumerate}\end{proposition}
 \begin{proof}
The first claim    is clear since it is immediate to check that  any function $f\in C(V)$ whose 
 restriction to 
 each cone of $\mathcal F$ is linear can written as
 $$f=\sum_cf(c)s_c.$$

 To see the second claim recall that for  any cone $C\in \mathcal F$,  its star $S(C)$  consists of cones in $\mathcal F$ having $C$ as a face. The function $s_c$ is clearly supported on  $S(c)$ (for brevity we write $S(c)$ instead than $S(C(c))$). From this it follows that  for any    monomial $m=x_{c_{i_1}}^{h_1}\cdots x_{c_{i_s}}^{h_{i_s}}$ the support of $\rho_{\mathcal F}(m)=s_{c_{i_1}}^{h_1}\cdots s_{c_{i_s}}^{h_{i_s}}$ is $S(C)$ if $C=C(c_{i_1}, \ldots ,c_{i_s})\in \mathcal F$ while $\rho_{\mathcal F}(m)=0$ otherwise. We deduce that $I_{\mathcal F}\subset \ker (\rho_{\mathcal F})$. In particular we obtain a homomorphism $\mu_{\mathcal F}:A_{\mathcal F}\to C(V).$ 
 
  Thus for  a monomial $m=x_{c_{i_1}}^{h_1}\cdots x_{c_{i_s}}^{h_{i_s}}$ with $C=C(c_{i_1}, \ldots ,c_{i_s})\in \mathcal F$, if $C'=C(c_{j_1}, \ldots ,c_{j_n})$ is a cone of maximal dimension,   $\mu_{\mathcal F}(mx_{c_{j_1}}\cdots x_{c_{j_n}})$  %\textcolor{blue}{$mx_{c_{j_1}}\cdots x_{c_{j_n}}$}
   is supported on $C'$ if $C$ is a face of $C'$, while $\mu_{\mathcal F}(m x_{c_{j_1}}\cdots x_{c_{j_n}})=0$ otherwise.
 
Take now a  polynomial  $P(x_{c_{j_1}},\cdots ,x_{c_{j_n}})$. The restriction of $\mu_{\mathcal F}(P(x_{c_{j_1}},\cdots ,x_{c_{j_n}}))$ to  $C'$ is just the evaluation of $P(x_{c_{j_1}},\cdots ,x_{c_{j_n}})$  hence it is zero if and only if $P(x_{c_{j_1}},\cdots ,x_{c_{j_n}})\equiv 0$.
 
 Take $a\in \ker \mu_{\mathcal F}.$ If $a\neq 0$ there is  an $n-$dimensional cone $\overline C=C(c_{i_1}, \ldots ,c_{i_n})$ such that 
 $b=ax_{c_{i_1}}\cdots x_{c_{i_n}}\neq 0$.  Then   $\mu_{\mathcal F}(b)$ is the restriction of a polynomial   $P(x_{c_{i_1}},\cdots ,x_{c_{i_n}})$ to $\overline C$. Hence it is zero if and only if    $b= 0$. A contradiction.

The last statement follows since, if we denote by $y_d$ the variable corresponding to a vertex $d$ of $\mathcal G$, for any vertex $c$ of $\mathcal F$
\begin{equation}\label{lamappa}\gamma_{\mathcal G}^{\mathcal F}(x_c)=\sum_{d \ \text{vertex of } \ \mathcal G}(\rho_{\mathcal F}(x_c)(d))y_d.\end{equation}
\textcolor{red}{}
as the reader can easily verify. \end{proof}

We now take a suitable algebra of continuous functions on $V$ 
\begin{definition} \begin{enumerate}\item The space $\Sigma$ consists of the functions $f\in C(V)$ such that 
\begin{enumerate}\item If  $\lambda \in X_*(T)$, $f(\lambda)\in \mathbb Q$.\item
There exists a rational smooth projective fan $\mathcal F$ such that for any $C\in \mathcal F$ the restriction of $f$ to $C$ is linear.
%\textcolor{red}{cambiato $\sigma$ in C}
\end{enumerate}
\item 
The equivariant ring of conditions $\C_T(T)$ is the $\mathbb Q$ subalgebra of the ring of continuous functions on $V$ generated by $\Sigma$.  \end{enumerate}
\end{definition}
In this paper we will always consider rational fans, so from now on the adjective `rational' will be omitted.
Since any two smooth projective fans admit a common refinement which is still  smooth and projective, it is clear that $\Sigma$ is a $\mathbb Q$-vector space. 

Let us order the set of $\mathfrak F$ of  smooth projective fans using refinement. We get a directed system $(A_{\mathcal F},  \gamma_{\mathcal G}^{\mathcal F})$. 
By Proposition \ref{conti} we deduce that $\C_T(T)$ is the union of the images of the homomorphisms  $\mu_{\mathcal F}$  and we  deduce (see \cite{Br}): 
\begin{proposition} $$\C_T(T)=\lim_{\mathcal F}A_{\mathcal F}.$$
\end{proposition}

Each element $\ell\in X^*(T)$ is a linear  function on $V$ taking integral values on $X_*(T)$ which, in terms of classes of the elements $x_{c_i}$, is  the class of
$$\sum_{i=1}^N\langle \ell, c_i\rangle x_{c_i}.$$
(we are taking into account the identification of $S_\mathcal F$ with $\rho_{\mathcal F} (S_\mathcal F)$).
In this way if we take a basis $\xi_1,\ldots ,\xi_n$ of $X^*(T)$ and set $R:=\mathbb Q[\xi_1,\ldots ,\xi_n]$,  $\mathcal A_{\mathcal F}$ is a free $R$ module (see \cite{Br} Corollary 1.3 and Proposition 2.2.) and we may consider the quotient algebra $$B_{\mathcal F}:=A_{\mathcal F}/(\xi_1,\ldots ,\xi_n)\simeq H^*(X_\mathcal F,\mathbb Q).$$ It is clear that the $\gamma_{\mathcal G}^{\mathcal F}$ induces an algebra homomorphism $\overline{\gamma}_{\mathcal G}^{\mathcal F}:B_{\mathcal F}\to B_{\mathcal G}$.
 
\begin{definition}
The ring of conditions for $T$ is the algebra
$$\mathcal C(T)=\lim_{\mathcal F}B_{\mathcal F}=\C_T(T)/(\xi_1,\ldots ,\xi_n).$$
\end{definition}

\section{Two differential graded algebras}
\label{sec:differentialgraded}

We now  want to define some differential graded algebras (DGA). Again we take a  projective smooth fan $\mathcal F$  in $V$.

%We start considering the exterior algebra $\bigwedge (\tau_1,\ldots \tau_N)$. We then define the differential graded algebra $(\D_{\mathcal F},d)$. 
We start with the  algebra  $\mathbb Q[x_c]\otimes \bigwedge (\tau_c)$, $c\in \Gamma_{\mathcal F}$.   We define a bigrading on  this algebra by setting  $\deg x_c=(2,0)$, $\deg\tau_c=(0,1)$. In general all the differential graded algebras we are going to consider will be easily seen to be bigraded, so we will often omit to specify how their bigrading is defined.

\begin{definition}  \begin{enumerate}\item The  algebra $D_{\mathcal F}$ is the quotient of the algebra $\mathbb Q[x_c]\otimes \bigwedge (\tau_c)$ modulo the bigraded ideal $\mathcal J_{\mathcal F}$ generated by ``the square free monomials"
$$x_{c_{i_1}}\cdots x_{c_{i_h}}\tau_{c_{j_1}}\cdots \tau_{c_{j_k}}$$
for each sequence of vertices $c_{i_1},\ldots ,c_{i_h},c_{j_1},\ldots ,c_{j_k}$ not spanning a cone in $\mathcal F$.
\item The differential $d$ is the unique derivation on $D_{\mathcal F}$ defined by
$$d(x_c)=0, \ \ \ d(\tau_c)=x_c.$$\end{enumerate}\end{definition}
Remark that $d$ preserves the relations in $D_{\mathcal F}$  and hence it is well defined and of degree 1.

It is easily seen that if for any cone $C=C(c_{i_1}, \dots ,c_{i_k})\in \mathcal F$, we take  the algebra 
$$A_{C,\mathcal F}:=\mathbb Q[x_{c_1},\ldots x_{c_N}]/[I_{\mathcal F}:(x_{c_{i_1}}\cdots x_{c_{i_k}})],$$
we have 
$$D_{\mathcal F}=\oplus_{C\in\mathcal F}A_{C,\mathcal F}\tau_{c_{i_1}}\cdots \tau_{c_{i_k}}.$$
and setting for each $m=0,\ldots ,n$ 
$$D_{m,\mathcal F}=\oplus_{C=C(c_{i_1},\ldots ,{c_{i_m})}}A_{C,\mathcal F}\tau_{c_{i_1}}\cdots \tau_{c_{i_m}},$$
the decomposition 
$$D_{\mathcal F}=\oplus_{m=0}^nD_{m,\mathcal F}.$$

\begin{proposition} \label{aciclico}
$$H^i(D_{\mathcal F},d)=\begin{cases} \mathbb Q \ \text{if} \  i=0\\ \ 0\ \text{if}\ i>0. \end{cases}$$\end{proposition}\begin{proof}
Let us consider the complex $(D^+_{\mathcal F},d)$ of elements of positive degree. We need to prove that this complex is exact. Let us define define a  map of degree $-1$
% \textcolor{blue}{mi sembra $+1$, come infatti \`e giusto}
$$S:D^+_{\mathcal F}\to D^+_{\mathcal F}$$
and show that $Sd+dS$ is the identity. This will give our claim. 

For this let fix a total order  $c_1,\ldots ,c_N$ of the vertices of $\mathcal F$. \\

Below, for brevity, we will write $x_{i_s}$, $\tau _{i_s}$ instead than $x_{c_{i_s}}$ , $\tau_{c_{i_s}}$. \\

Let $m=x_{i_1}^{h_1}\cdots x_{i_s}^{h_s}\tau_{j_1}\cdots \tau_{j_r}\in D^+_{\mathcal F}$, with $j_1<j_2\cdots <j_r$ and if $s>0$, $h_i>0$.  If $s=0$ (resp.$r=0$) we set $i_1=\infty$  (resp.$j_1=\infty$).  Notice that $r+s>0$ and set $f=\min (i_1,j_1)$. We define

$$S(m)=\begin{cases}0 \ \ \text{if}\ f=j_1\\  x_{i_1}^{h_1-1}\cdots x_{N}^{h_N}\tau_{i_1}\tau_{j_1}\cdots \tau_{j_r} \ \ \text{if}\ f=i_1<j_1\end{cases}$$
%Remark that if $j_1=1$ then $S(m)=0$.
 Now let us compute 
 $(dS+Sd)(m)$.
 If $f=j_1$ we have $dS(m)=0$ and 
 $$Sd(m)=S(x_{j_1}x_{i_1}^{h_1}\cdots x_{i_s}^{h_s}\tau_{j_2}\cdots \tau_{j_r} +\sum_{\ell=2}^r (-1)^{\ell+1}x_{j_\ell}x_{i_1}^{h_1}\cdots x_{i_s}^{h_s}\tau_{j_1}\cdots \check{\tau}_{j_\ell}\cdots \tau_{j_r})=m.$$
% \textcolor{blue}{confesso che continuo a pensare che nel mezzo della formula qui sopra  per come e' scritta ci voglia il +.. in ogni caso non e' importante perche' tanto poi S manda qualla parte a 0 }
 
If  $f=i_1<j_1$, one easily sees that 
 $$dS(m)=m- \sum_{\ell=1}^r (-1)^{\ell+1}x_{j_\ell}x_{i_1}^{h_1-1}x_{i_2}^{h_2}\cdots x_{i_s}^{h_s}\tau_{i_1}\tau_{j_1}\cdots \check{\tau}_{j_\ell}\cdots \tau_{j_r}=m-Sd(m)$$
 and everything  follows.
\end{proof}

%Notice that for $A_{\{0\}}=A_\mathcal F$. 
%%From now on we fix a total ordering on the set of primitive on the set of primitive vectors in $X_*(T)$.
%  More generally the ring $A_C$ is the equivariant cohomology ring  of the closure of the $T$ orbit in $X_{\mathcal F}$  corresponding to the cone $C$.
%  

%ORA QUA BISOGNA FARE IL LIMITE ED INTERPRETARLO. QUI MI CONFONDO ANCHE SE LA COSA SE FUNZIONE DOVREBBE ESSERE FACILE
%PER ORA POCHE RIGHE

When the fan $\mathcal G$ is a (smooth, projective) refinement of $\mathcal F$, we want to compare the algebras $D_{\mathcal F}$ and   $D_{\mathcal G}$.

As before in order to avoid confusion for $d\in \Gamma_{\mathcal G}$ we denote by $y_d$ and $\upsilon_d$ the corresponding even and odd variables.
We   define 
a homomorphism  
$$\xi_{\mathcal G}^{\mathcal F}:\mathbb Q[x_c]\otimes \bigwedge (\tau_c)\to\mathbb Q[y_d]\otimes \bigwedge (\upsilon_d)$$
by setting 
\begin{equation}\label{lamappa1}\xi_{\mathcal G}^{\mathcal F}(x_c)=\sum_{d \ \text{vertex of } \ \mathcal G}(\rho_{\mathcal F}(x_c)(d))y_d.\end{equation}
\begin{equation}\label{lamappa2}\xi_{\mathcal G}^{\mathcal F}(\tau_c)=\sum_{d \ \text{vertex of } \ \mathcal G}(\rho_{\mathcal F}(x_c)(d))\upsilon_d.\end{equation}
We then set
 $$  \overline \xi_{\mathcal G}^{\mathcal F}=q\circ \xi_{\mathcal G}^{\mathcal F}:\mathbb Q[x_c]\otimes \bigwedge (\tau_c)\to   D_{\mathcal G},$$
 $q$ being the quotient modulo ${\mathcal J}_{\mathcal G}$.
We then have
%Notice that $\zeta_{\mathcal G}^{\mathcal F}(x_c)$ and $\zeta_{\mathcal G}^{\mathcal F}(x_c)(\tau_c)$  are linea combinations with non negative coefficients of the $y_d$'s and $\upsilon_d$'s respectively.
\begin{proposition}  
$\overline \xi_{\mathcal G}^{\mathcal F}(\mathcal J_{\mathcal F})=0$. It follows that $\overline \xi_{\mathcal G}^{\mathcal F}$ factors through a   homomorphism  of differential graded algebras
$$  \zeta_{\mathcal G}^{\mathcal F}: D_{\mathcal F}\to  D_{\mathcal G}.$$
\end{proposition}

\begin{proof} 

Let us take a monomial $x_{c_{i_1}}\cdots x_{c_{i_h}}\in \mathcal J_{\mathcal F}$ that is $c_{i_1},\ldots ,c_{i_h}$ do not span a cone in $\mathcal F$. We know by Proposition \ref{conti} that 
$\overline \xi_{\mathcal G}^{\mathcal F}(x_{c_{i_1}}\cdots x_{c_{i_h}})=0$.

Now notice that $\xi_{\mathcal G}^{\mathcal F}(x_c)$ is a linear combination with non negative coefficients of the $y_d$.
We deduce that
$ \xi_{\mathcal G}^{\mathcal F}(x_{c_{i_1}}\cdots x_{c_{i_h}})$
is a linear combination with non negative coefficients of monomials in the $y_d$. Thus each monomial appearing with non zero coefficient has to  lie in  $ \mathcal J_{\mathcal G}.$

Necessarily if in any such monomial we substitute some of the $y_d$'s with the corresponding $\upsilon_d$'s we also get a relation in $ D_{\mathcal G}.$

This immediately implies that  for any $h>0$, $\xi_{\mathcal G}^{\mathcal F}(\tau_{c_{i_{1}}}\cdots \tau_{c_{i_h}})\in \mathcal J_{\mathcal G}$ and for any $1\leq \ell\leq h$,
$ \xi_{\mathcal G}^{\mathcal F}(x_{c_{i_1}}\cdots x_{c_{i_\ell}}\tau_{c_{i_{\ell+1}}}\cdots \tau_{c_{i_h}})\in \mathcal J_{\mathcal G}$.

The fact that $d\circ \zeta_{\mathcal G}^{\mathcal F}=\zeta_{\mathcal G}^{\mathcal F}\circ d$  then follows  from the definitions.
\end{proof}
%\textcolor{blue}{qui nel finale ho cambiato alcune $ \zeta$ in $xi=$, chiamando $zeta$ solo la mappa finale} \D

Notice now that it is clear that $  \zeta_{\mathcal G}^{\mathcal F}$ is a quasi isomorphism so that setting
$$( D,d)=\lim_{\mathcal F}( D_\mathcal F,d)$$
we deduce that 
$$H^i( D,d)=\begin{cases} \mathbb Q \ \text{if} \  i=0\\ \ 0\ \text{if}\ i>0. \end{cases}$$

Finally, let us remark that   $\zeta_{\mathcal G}^{\mathcal F}(\ D_{m,\mathcal F})\subset \ D_{m,\mathcal  G}$ for each $m=0,\ldots n$. It follows that, taking the limit, 
$D_m=\lim_{\mathcal F}( D_{m,\mathcal F})$ we get a direct sum decomposition
$$ D=\bigoplus_m D_m.$$
In particular for $m=0$, $D_0=B_T(T)$. We denote by $\nu_{\mathcal F}:D_\mathcal F\to D$ the natural morphism.

We want to give a more explicit description of the algebra $\ D$. In order to do so let us take the exterior algebra $\bigwedge(\Sigma)$. Define the algebra $\E=B_T(T)\otimes \bigwedge(\Sigma)/H$, where $H$ is the ideal generated by the elements $s_1\cdots s_t\otimes \sigma_1\cdots \sigma_r$, with $s_i,\sigma_j\in \Sigma$,  such that $s_1\cdots s_t \sigma_1\cdots \sigma_r=0$ in $B_T(T)$.
Notice that the natural differential $d$ on $B_T(T)\otimes \bigwedge(\Sigma)$ defined by setting $d(a\otimes s)=as\otimes 1\in B_T(T)$ and extended  as an algebra derivation, clearly preserves $H$. It follow that   we get a differential on  $\E.$ We claim,
\begin{proposition} $\E\simeq D$ as differential graded algebras.
\end{proposition}
\begin{proof} For our usual  smooth projective fan $\mathcal F$, we have already defined a map $\mu_{\mathcal F}:A_{\mathcal F}\to B_T(T)$,  with the property that  $\mu_{\mathcal F}(x_c)=s_c$ for any ray $c$ of $\mathcal F$, which gives an inclusion of $S_{\mathcal F}$ into  $ \Sigma$ and hence a map $\bigwedge(S_{\mathcal F})\to  \bigwedge(\Sigma)$. Tensoring, we obtain a map
$$A_{\mathcal F}\otimes \bigwedge(S_{\mathcal F})\to B_T(T)\otimes \bigwedge(\Sigma)$$ and composing with the quotient, a  map
$$A_{\mathcal F}\otimes \bigwedge(S_{\mathcal F})\to \E$$
By the very definition of $ D_{\mathcal F}$ we deduce that this map factors through a map $$\overline\nu_\mathcal F:D_{\mathcal F}\to \E$$
Passing to the limit and recalling that $\Sigma$ is spanned by the functions $s_c$ for some ray in a suitable fan, we get a  surjective  map
$$\mu: D\to \E$$
On the other hand, if we take an element $s_1\cdots s_t\otimes \sigma_1\cdots \sigma_r$, we can find a fan $\mathcal F$ with the required properties, such that  each $s_i=\mu_{\mathcal F}(x_i)$ and each $\sigma_j=\mu_{\mathcal F}(\tau_j)$ with  $x_i,\tau_j\in S_\mathcal F\setminus \{0\}$. Thus we can map $s_1\cdots s_t\otimes \sigma_1\cdots \sigma_r$ to $\nu_\mathcal F(x_1\cdots x_t\otimes \tau_1\cdots \tau_r)$. In order to see that this map is well defined we just have to show that if 
the function $s_1\cdots s_t \sigma_1\cdots \sigma_r =0$ in $C(V)$, the element $x_1\cdots x_t\ \tau_1\cdots \tau_r=0$ in $A_\mathcal F$. 

If we write each $x_i$ and each $\tau_j$ as a linear combination of the basis elements $x_c$, $c$ a ray of $\mathcal F$, of $S_{\mathcal F}$ we get that $x_1\cdots x_t\ \tau_1\cdots \tau_r$ is the image of a polynomial $P(x_c)\in \mathbb Q[x_c]$ which is a product of non zero linear functions and hence non zero. We deduce that if we write $P(x_c)$ as a linear combination of monomials and we compute it as a function on $V$, we get $0$ if and only if each monomial appearing with non zero coefficient in $P(x_c)$ is zero in $A_\mathcal F$ hence the element $x_1\cdots x_t\otimes \tau_1\cdots \tau_r=0$ in $A_\mathcal F$.

It follows that we get a map $\E\to D$ and it is immediate to check that this map is the inverse of $\mu$.
\end{proof}
\bigskip

For every cone $C\in \mathcal F$, $A_{C,\mathcal F}$ is a 
 $R$-module. Again by Corollary  1.3 and Proposition 2.2. in \cite{Br}, $A_{C,\mathcal F}$  is free of rank equal to the number of $n$ dimensional cones in the star $S(C)$ of $C$. Furthermore $A_{C,\mathcal F}$  is isomorphic to  the $T$-equivariant cohomology of the closure $X_{C,\mathcal F}$ of the $T$-orbit associated to the cone $C\in \mathcal F$ and the 
  quotient algebra $$B_{C,\mathcal F}:=A_{\mathcal F}/(\xi_1,\ldots ,\xi_n)\simeq H^*(X_\mathcal F,\mathbb Q).$$ From this we deduce in particular that $ D_{\mathcal F}$ is a free $R$ module and, setting by abuse of notation, $\xi_i:=\xi_i\otimes 1$, for each $i=1,\ldots ,n$,  we may consider the quotient algebra $$\mathcal C_{\mathcal F}:=\ D_{\mathcal F}/(\xi_1,\ldots ,\xi_n).$$
  
Since each element $\xi_j$ is a cocycle, we deduce that the ideal $(\xi_1,\ldots ,\xi_n)$ is preserved by the differential $d$ and we have an induced differential on $\mathcal C_{\mathcal F}$ which we shall denote by the same letter.

 Notice that if we set in  $\ D_{\mathcal F}$
$$\psi_j=\sum_{i=1}^N\langle \ell_j, c_i\rangle \tau_{c_i}.$$
 in  $\mathcal C_{\mathcal F}$ we get that $d(\psi_j)=0$ so that we obtain an inclusion of the exterior algebra $\bigwedge (\psi_1,\ldots, \psi_n)$ into the subalgebra $Z(\mathcal C_{\mathcal F}$) of cocycles and 
 a degree preserving homomorphism $$j_{\mathcal F}:H^*(T)\to H^*(\mathcal C_{\mathcal F}),$$
defined by setting $j_{\mathcal F}(\ell_j)=\psi_j$.

We have 
\begin{proposition}  The homomorphism $j_{\mathcal F}$ is an isomorphism.
\end{proposition}
\begin{proof} We shall deduce this by induction from a slightly more general fact. For any $0\leq h\leq n$ consider 
$$\mathcal C_{\mathcal F}^{(h)}=\begin{cases} \D_{\mathcal F} \ \ \text{ if}\  h=0\\  \D_{\mathcal F}/(\xi_1,\ldots ,\xi_h) \ \ \text{ if}\  h>0\end{cases}.$$

Our claim is that for every $h$, $H^*(\mathcal C_{\mathcal F}^{(h)})\simeq \bigwedge (\psi_1,\ldots ,\psi_h)$.

For $h=0$, $\mathcal C_{\mathcal F}^{(h)}=\D_{\mathcal F}$ and our claim is Proposition \ref{aciclico}.

We proceed by induction on $h$ and assume the claim proved for $h-1$. By reasoning as above we deduce that $\psi_h$ is a cocycle in $Z(\mathcal C_{\mathcal F}^{(h)})$ and hence gives a class in $H^1(\mathcal C_{\mathcal F}^{(h)})$. 
%\textcolor{blue}{qui era  cambiata la notazione degli indici. Rimetto le $(h)$ in alto sia nel rigo sopra sia nel prosieguo della dimostrazione}

Since clearly $\xi_h$ is a non zero divisor in $\mathcal C_{\mathcal F}^{(h-1)}$ we take the exact sequence
%\begin{equation}\label{prod} \CD 
%A_{C_1}\otimes A_{C_2} @         >\phi_{C}^{C_1}\otimes \phi_{C}^{C_2} >>      A_{C}\otimes A_{C} @> m>>      A_{C} , \endCD \end{equation}
\begin{equation}\label{prod} \CD 0\to \mathcal C_{\mathcal F}^{(h-1)}[-2]@         >\circ \xi_h>>  \mathcal C_{\mathcal F}^{(h-1)}@         > >>\mathcal C_{\mathcal F}^{(h)}\to 0.\endCD \end{equation} the corresponding long exact sequence in cohomology and using the $\xi_h$ being a coboundary induces the trivial homomorphism in cohomology, we deduce the exact sequence
\begin{equation}\label{exacta} \CD 0\to H^*( \mathcal C_{\mathcal F}^{(h-1)})\to H^*( \mathcal C_{\mathcal F}^{(h)})\to H^{*-1}( \mathcal C_{\mathcal F}^{(h-1)})\psi_h\to 0.\endCD \end{equation}
Since by induction $H^*( \mathcal C_{\mathcal F}^{(h-1)})\simeq \bigwedge (\psi_1,\ldots ,\psi_{h-1})$, this implies our claim.
\end{proof}
We finish by remarking that clearly the map $  \zeta_{\mathcal G}^{\mathcal F}$ is a map of $R$-modules, so it induces a map 
$$\chi_{\mathcal G}^{\mathcal F}:\mathcal C_{\mathcal F}\to \mathcal C_{\mathcal G}$$
and we may also consider \begin{equation}\label{defini}\mathcal C=\D/(\xi_1,\ldots ,\xi_n) =\lim_{\mathcal F} \mathcal C_{\mathcal F}.\end{equation}
It is immediate to see that $\mathcal C$ inherits a direct sum decomposition
$$\mathcal C=\bigoplus_m \mathcal C_m,$$
with $\mathcal C_m=\lim_{\mathcal F} \mathcal C_{m, {\mathcal F}}=\lim_{\mathcal F} \D_{m, {\mathcal F}}/(\xi_1,\ldots ,\xi_n)$.
In particular for $m=0$  $$\mathcal C_0=\mathcal C(T)=\mathcal B_T(T)/(\xi_1,\ldots ,\xi_n) $$  is the ring of conditions of the torus $T$.
%\textcolor{red}{qui era rimasto un $C_T(T)$}

Furthermore the $\chi_{\mathcal G}^{\mathcal F}$ are quasi isomorphism and we deduce that
also for $\mathcal C$ we have 
$$H^*(\mathcal C,d)\simeq \bigwedge (\psi_1,\ldots ,\psi_n)\simeq H^*(T).$$

\section{Toric arrangements}
\label{sec:toric}
Let us now recall from the Introduction the definition of a toric arrangement.  
A {\em layer} in \(T\) is the subvariety 
$$\mathcal K_{\Gamma,\phi}=\{t\in T|\, \chi(t)=\phi(\chi),\, \forall \chi\in \Gamma\}$$
where $\Gamma$ is a split direct summand of  $X^*(T)\cong \Z^n$ and  $\phi:\Gamma \to \mathbb C^*$ is a homomorphism.

A  toric arrangement \(\A\) is given by   finite set of  layers \(\A=\{\mathcal K_{1},...,\mathcal K_{m}\}\) in $T$.

In \cite{DCG1}  it is shown   how to construct  {\em projective wonderful models} for  the complement \(\emme(\A)=T-\bigcup_i \mathcal K_i\).  

A projective wonderful model    is a smooth projective  variety \ containing \(\emme(\A)\) as an open set and such that the complement of  \(\emme(\A)\) is  a divisor with normal crossings and smooth irreducible components. 
As we mentioned in the Introduction, we have  \(\emme(\A)=\emme(\tilde \A)\), where \(\tilde \A\) is the saturation of \(\A\), i.e. the arrangement consisting of all the  layers which are obtained as connected components of intersections of layers in $\mathcal A$.

Therefore from now on, for brevity of notation, we are going to assume \(\A=\tilde \A\).
%We recall that the problem of finding a  wonderful model for \(\emme(\A)\) was first studied by Moci in \cite{mociwonderful}, where a construction of  non projective  models was described.

%
%In this paper in order to make our life easier we are going to assume that  \(\A=\{\mathcal K_{1},...,\mathcal K_{m}\}\)   has the following property.
%
%\begin{property} \label{chiudi}
%
%For any subset \(\{\mathcal K_1,...,\mathcal K_k\}\) in \(\A\),  the intersection \(\mathcal K_1\cap \mathcal K_2\cap ...\cap \mathcal K_k\) is either empty  or it is the union of connected components each belonging to \(\A\).
%\end{property}
%
%
%Remark that   if we take a finite set  of layers \(\mathcal H_{1},...,\mathcal H_{s}\), with non empty intersection,   each connected component of  \(\mathcal H_1\cap \mathcal H_2\cap ...\cap \mathcal H_s\)  is itself a layer. It follows that if we start with an arbitrary   toric arrangement  \(\A\) and we add to it all the non empty    connected components of   intersections of subsets of \(\A\), we get a new toric arrangement  \(\overline\A\) which satisfies Property \ref{chiudi} and is such that  \(\emme(\A)= \emme(\overline\A)\).

Let us put $V=\hom_\Z(X^*(T),\R)= X_*(T)\otimes_{\Z} \R$.
A layer $\mathcal K_{\Gamma,\phi}$ is  a coset with respect to the torus 
$T_\Gamma=\cap_{\chi\in \Gamma}Ker(e^{2\pi\imath \chi})$, and  we can consider  
  the subspace $$V_\Gamma=\{v\in V|\, \langle\chi,v\rangle=0,\, \forall \chi\in \Gamma\}.$$
   Since $X^*(T_\Gamma)=X^*(T)/\Gamma$, $V_\Gamma$ is naturally isomorphic to $\hom_\Z(X^*(T_\Gamma),\R)= X_*(T_\Gamma)\otimes_{\Z} \R$.

  \begin{definition} Let $\mathcal F$ be a fan in $V$. A finite set   $\{\chi_1,\ldots ,\chi_s\}$ of vectors in $X^*(T)$ is said to have equal sign with respect to $\mathcal F$ if for each $i=1,\ldots ,s$ and each cone $C\in \mathcal F$, the function 
  $\langle\chi_i,-\rangle$ has constant sign on $C$, i.e. it is either non negative or non positive on $C$.\end{definition}

In \cite{DCG1} (see Proposition 6.1)  it was shown how to construct a projective smooth   \(T\)-embedding  \(X_{\mathcal F}\) whose fan \({\mathcal F}\) in \(V\) has the following property. For every \(\Gamma_i\)   there is an integral  basis of $\Gamma_i$, $\chi_1,\ldots ,\chi_s$, which has   equal sign with respect to $\mathcal F$.  The basis $\chi_1,\ldots ,\chi_s$ is called an {\em equal sign basis} for \(\Gamma_i\).

In fact by the same proof one can even show that one can construct $\mathcal F$  such that for any pair of layers  $\mathcal K_{\Gamma,\phi}\subset \mathcal K_{\Gamma',\psi}\in \A$,   
there is an equal sign basis for 
$\Gamma$ whose  intersection with $\Gamma'$ is an equal sign basis for $\Gamma'$. 

In view of this we define

\begin{definition} Let  \(\A\) be a toric toric arrangement.  A  smooth  projective fan $\mathcal F$  is compatible with \(\A\) if for any pair of layers  $\mathcal K_{\Gamma,\phi}\subset \mathcal K_{\Gamma',\psi}\in \A,$   
there is an equal sign basis for 
$\Gamma$ whose  intersection with $\Gamma'$ is an equal sign basis for $\Gamma'$. \end{definition}

In what follows we are always going to consider     fans $\mathcal F$ compatible with \(\A\).
%\textcolor{blue}{ho visto la nuova definizione, non mi vengono in mente parole migliori di compatible...forse un simile $\mathcal F$ potrebbe chiamarsi $\A$- monotone oppure  $\A$ - equalized..}
Once such $\mathcal F$ has been constructed,  the  strategy used in \cite{DCG1}  is to  first embed the torus \(T\)  in      \(X_{\mathcal F}\).

%One then shows that in this toric variety  the set made  by  the connected components of the intersections of the closures of the layers of  \(\A\) turns out to be  an arrangement of subvarieties \(\mathcal L'\)  in the sense of Li  \cite{li}. In turn this allows to   apply Li's construction in order to get a projective wonderful model.

%Moreover we remark that  \({\mathcal F}\) can be chosen in such a way that for every layer $\mathcal K_{\Gamma,\phi}$, obtained as a connected component of the intersection of some of the layers in \(\A\), the lattice \(\Gamma\) has an equal sign basis. Given such a \({\mathcal F}\), we will say that \(X=X_{\mathcal F}\) is a {\em good toric variety} for  \(\A\).

In such a toric variety  \(X_{\mathcal F}\)  consider the closure  \(\overline {{\mathcal K}}_{\Gamma,\phi}\) of a layer. This  closure turns out to be a toric variety, whose explicit description is provided by  \cite{DCG1}. 
%\end{document}

%\textcolor{blue}{nel teorema qui sotto il fan era rimasto $\Delta$, ho messo $\mathcal F$ per coerenza con le notazioni di questo articolo}
\begin{theorem}[Proposition 3.1 and Theorem 3.1 in \cite{DCG1}]
\label{faneH} For every layer    \(\mathcal K_{\Gamma,\phi}\), let \(T_\Gamma\) be the corresponding subtorus and let $V_{\Gamma}=\{v\in V|\, \langle\chi,v\rangle=0,\, \forall \chi\in \Gamma\}$. Then, 
\begin{enumerate}
\item For every cone $C\in \mathcal F$, its relative interior is either entirely contained in $V_{\Gamma}$ or disjoint from $V_{\Gamma}$.
\item The collection of cones $C\in \mathcal F$ which are contained in $V_{\Gamma}$ is a smooth fan $\mathcal F_{\Gamma}$.

\item   \(\overline {{\mathcal K}}_{\Gamma,\phi}\) is a smooth $T_\Gamma$-variety whose fan is $\mathcal F_{\Gamma}$.
\item  Let $\mathcal O$ be a $T$ orbit in $X:=X_{\mathcal F}$ and let $C_{\mathcal O}\in \mathcal F$ be the corresponding cone. Then 
\begin{enumerate}
\item If $C_{\mathcal O}$ is not contained in $ V_{\Gamma}$, $\overline {\mathcal O}\cap \overline {{\mathcal K}}_{\Gamma,\phi}=\emptyset$.
\item If $C_{\mathcal O}\subset V_{\Gamma}$, $ {\mathcal O}\cap \overline {{\mathcal K}}_{\Gamma,\phi}$ is the $T_\Gamma$ orbit in $\overline {{\mathcal K}}_{\Gamma,\phi}$ corresponding to 
$C_{\mathcal O}\in \mathcal F_{\Gamma}$. 
\end{enumerate} \end{enumerate}\end{theorem}

Let us   denote  by \({\mathcal Q}'\) (resp.  \({\mathcal Q}\) )  the set whose elements are  the subvarieties  \(\overline {{\mathcal K}}_{\Gamma_i,\phi_i}\)   of \(X_{\mathcal F}\) (resp. the subvarieties  \(\overline {{\mathcal K}}_{\Gamma_i,\phi_i}\)  and the irreducible components of the complement \(X_{\mathcal F}- T\)). We then denote by \(\elle'\) (resp. \(\elle\)) the poset  made by all the connected components of all the intersections of some of the  elements of \({\mathcal Q}'\) (resp.  \({\mathcal Q}\) ).
In \cite{DCG1} (Theorem 7.1) we have shown that the  family   \(\elle\) is an arrangement  of subvarieties in \(X_{\mathcal F}\) in the sense of Li's paper \cite{li}. As a consequence also \(\elle'\), being contained in $\elle$ and closed under intersection, is an arrangement  of subvarieties.

%As a consequence also \(\elle'\), being contained in $\elle$ and closed under intersection, is an arrangement  of subvarieties.

Let   \(\elle'=\{G_1,...,G_m\}\),  ordered in such a way  that   if \(G_i\subsetneq  G_j\)  then \(i<j \). Thus for each $i=1,\ldots ,m$ we have $G_i=\overline {\mathcal K}_{\Gamma_i,\phi_i}$ for  a suitable  pair $(\Gamma_i,\phi_i)$.

A this point, following Li's construction for \(\elle'\)  we  construct the variety    \(Y(X_{\mathcal F})\), which  is a projective wonderful model for  \(\emme(\A)=X_{\mathcal F}-\bigcup_{A\in \elle}A.\) 
This means that    \(Y(X_{\mathcal F})\) contains  \(\emme(\A)\) as a dense open set whose
complement is a divisor with smooth irreducible components having transversal intersections.

%\textcolor{blue}{inizio pezzo nuovo}

More in detail we choose \(\elle'\) as a  {\em building set} (see Definition 2.5 in \cite{DCG2}). Then we   obtain  \(Y(X_{\mathcal F})\) starting from \(X_{\mathcal F}\) and  blowing up the elements of   \(\elle'\) (after the first step, their  transforms) in any order such that if \(G_{i_1}\subset G_{i_2}\) we blow up (the  transform of ) \(G_{i_1}\) before (the transform of ) \(G_{i_2}\). In particular we notice that the ordering we chose in \(\elle'\)  is one of the admissible orderings  to perform  these blowups.

%is such that if $i_1,G_{i_2}\in \elle'$ and  \(G_1\subset G_2\) we blow up (the  transform of ) \(G_1\) before (the transform of ) \(G_2\).
%We notice that the complement in \(X\) of the union of the elements in \(\elle\) is equal to \(\emme(\A)\), and it is strictly contained in the complement of the union of the elements in \(\elle'\).

In Proposition 5.2 of \cite{DCG2} we observed  that  \(Y(X_{\mathcal F})\) is isomorphic to the variety \(Y^+(X_{\mathcal F})\) obtained by choosing as a building set the set \(\elle^+=\elle' \cup \{D_{c_i}\}_{i=1,...,N}\) where for every vertex \(c_i\) of \(\mathcal F\),  \(D_{c_i}\) is the associated irreducible divisor  in the boundary of \(X_{\mathcal F}\).  The isomorphism is an immediate consequence  of the fact that the  \(D_{c_i}\)'s (and hence their transforms in   \(X_{\mathcal F}\)) are   divisors.

From Theorem 1.2 in \cite{li} it follows that  \(Y^+(X_{\mathcal F})\setminus \emme(\A)\)  is a divisor with normal crossings whose   irreducible components are smooth and indexed by \(\elle^+\). For $j=1,\ldots m$, we denote by $D_{G_j}$ the component of \(Y^+(X_{\mathcal F})\setminus \emme(\A)\) corresponding to $G_i$ and, by abuse of notation,  we still denote by \(D_{c_i}\) its transform in \(Y^+(X_{\mathcal F})\), so that
$$Y^+(X_{\mathcal F})\setminus \emme(\A)=(\cup_{j=1}^m D_{G_j})\cup(\cup_{c_i}D_{c_i}).$$
It follows from the theory of torus embeddings that for a collection of rays $c_{i_1},\ldots c_{i_t}$ the intersection $\cap_{h=1}^tD_{c_{i_h}}$ is non empty if and only if $C=C( c_{i_1},\ldots ,c_{i_t})$ is a cone in $\mathcal F$.

Furthermore from  the  general definition of nested set  (see  Definition 5.6 of \cite{li} and also Definition 2.7 in \cite{DCG2}),    one can easily check that in our special situation,  if we take a subset  $\underline G=\{G_{j_1}, G_{j_2}, \ldots, G_{j_s}\}$ of \(\elle'\) and   a cone  $C=C(c_{i_1},\ldots c_{i_t})$, the 
intersection 
$$Y_{(\underline G, C)}:=(\cap_{k=1}^sD_{G_{j_k}})\cap(\cap_{h=1}^tD_{c_{i_h}})$$
is non empty if and only if 
$G_{j_1}\subsetneq G_{j_2}\subsetneq \cdots \subsetneq G_{j_s}$ and 
%. Furthermore by Theorem \ref{faneH} it follows that the nested sets of  \(\elle^+\) are indexed  by the couples $(\underline G, C)$ where   $\underline G$ is as above  and  $C\in \mathcal F$ is such that   
 $C\subset V_{G_{j_1}}$. In this case $Y_{(\underline G, C)}$ is smooth and irreducible.
% In particular, the component in the boundary of \(Y^+(X_{\mathcal F})\) which corresponds to $(\underline G, C)$ is the intersection of the transforms of all the  subvarieties in the set $\{G_{i_1}, G_{i_2}, \ldots, G_{i_s}\}\cup \{D_{c_i}\}_{c_i\in C}\).

\begin{remark}
From now on we will identify \(Y(X_{\mathcal F})\) and \(Y^+(X_{\mathcal F})\).
\end{remark}
%\textcolor{blue}{fine pezzo nuovo}
%\textcolor{blue}{inizio pezzo nuovo}

%\textcolor{blue}{bisogna specificare meglio che il complementare di  \(\emme(\A)\) \`e un divisore con normal crossing? forse si potrebbe aggiungere il teorema di descrizione del bordo}
% \textcolor{red}{FALLO TU. In effetti visto il caso particolare mi pare che gli strati del bordo siano indicizzati da coppie $(\underline G, C)$con $\underline G=(G_{i_1}\subsetneq G_{i_2}\subsetneq \cdots \subsetneq G_{i_s})$ e $C\in \mathcal F$ con 
% $C\subset V_{G_{i_1}}$. 
%  \`E giusto? }
%  
%  \textcolor{blue}{concordo, anche se probabilmente quando lavoriamo col bordo sar\`a bene aggiungere una frase per  specificare esplicitamente qual \`e il building set che stiamo considerando. Mi sembra che che stiamo considerando i nested sets nel building set  $\elle'$ unito i divisori $D_r$ con $r$ raggio, come si era fatto nell'articolo precedente}
%  
%  
%Let   \(\elle'=\{G_1,...,G_m\}\),  ordered in such a way  that   if \(G_i\subsetneq  G_j\)  then \(i<j \). Thus for each $i=1,\ldots ,m$ we have $G_i=\overline {\mathcal K}_{\Gamma_i,\phi_i}$ for  a suitable  pair $(\Gamma_i,\phi_i)$.

In \cite{DCG2} we have described  the cohomology ring \(H^*(Y(X_{\mathcal F}),\Z)\) 
%(which is isomorphic to \(H^*(Y^+(X_{\mathcal F}),\Z)\))
 by generators and relations in a greater generality. Here we shall illustrate this result under our assumption, leaving the straightforward translation to the reader.
We refer to \cite{DCG2} for the geometric explanation of our relations.

To simplify notation we are going to add to \(\elle'\) the element $G_{m+1}:=X_{\mathcal F}$.
%\textcolor{red}{metterei $X_{\mathcal F}$ liberandosi di $X$}
We need to introduce certain polynomials in  $B_{\mathcal F}[t_1,\ldots ,t_m]$.

%Take a pair   $(i,A)$ with   \(i\in \{1,...,m\}\), and    \(A\subset \{1,...,m\}\) such that if \(j\in A\) then \( G_i\subsetneq G_j\).

Take a pair   $(i,j)$ with   \(i\in \{1,...,m\}\), and    \(j\in \{1,...,m+1\}\) in such a way that  \( G_i\subsetneq G_j\).
Consider the set  \(B_i=\{h\: | \: G_h\subseteq G_i\}\).

%Take 
% $\Gamma_j}$ (if $j=m+1$, clearly $\Gamma_{m+1}=\{0\}$)
    Take an equal sign basis $\underline \chi$  of $\Gamma_i$ whose intersection with  $\Gamma_j$ (if $j=m+1$,  $\Gamma_{m+1}=\{0\}$) is a basis of $\Gamma_j$. We then set
%$$P^{G_j}_{G_i}(t):=\prod_{\chi\in \underline \chi\setminus (\underline \chi\cap \Gamma')}(t-\chi-|{\mathcal F}^-)\in B_{\mathcal F}[t]$$
$$P^{G_j}_{G_i}(t):=\prod_{\chi\in \underline \chi\setminus (\underline \chi\cap \Gamma_j)}(t-\chi_{\mathcal F}^-)\in B_{\mathcal F}[t]$$
%\textcolor{red}{ma chi \`e quel $\Gamma'$? il simbolo sbagliato forse nasce dal discorso che c'\`e in fondo a pagina 9, in realt\`a qui volevamo mettere $\Gamma_j$} 
with
$$\chi_{\mathcal F}^-=\sum_{c\ \text{ray}}\min(0,\langle \chi,c\rangle) x_c,$$
and   
  \[ F(i,j)=P^{G_j}_{G_i}(\sum_{h\in B_i}-t_{h})t_{j}, \]  with $t_{m+1}:=1$.
%  \textcolor{blue}{nella formula sopra ho aggiunto una parentesi per racchiudere insieme $(\sum_{h\in B_i}-t_{h})$}

%We also include as  special cases the pairs \((0,A)\)  where 
%%$G_0=\emptyset$  and 
%$A$ is such that  \(\bigcap_{j\in A}G_j=\emptyset \), and we define the polynomials:  
%\[ F(0,A)=\prod_{j\in A}t_{j}.\]
%
%
%Let us now consider $G:=\overline{\mathcal K}_{\Gamma,\phi}\in \mathcal L'$. 
%%% obtained as a connected component of the intersection of the closures of some of the layers in \(\A\),  
%%we set  $\Lambda_G:=\Gamma$.}
From    \cite{DCG2}   we easily get 

%we show the following result (in fact our result holds over the integers but we state here with rational coefficients).
%
%For any cone $C\in\mathcal F$, let us set $B_{C,\mathcal F}:=A_{C,\mathcal F}/(\xi_1,\ldots ,\xi_n)$ this is the cohomology ring of the orbit closure of the $T$-orbit in $X_{\mathcal F}$ associated to $C$. For simplicity
%we set $B_{\mathcal F}:=B_{\{0\},\mathcal F}$ the cohomology ring of $X_{\mathcal F}$.

\begin{theorem}[Proposition 6.3 and Theorem 7.1 in \cite{DCG2}]
\label{propindependent}
Let  $I$ be the ideal in  $ B_{\mathcal F}[t_1,\ldots ,t_m]$  generated by  \begin{enumerate}\item the products 
$t_ix_c$ for every  ray $c\in  {\mathcal F} $  that does not belong to \(V_{\Gamma_i}\).
\item  the products 
 $t_st_r$ if $G_s$ and $G_r$ are not comparable.
\item the polynomials   $F(i,j)$, for $G_i\subsetneq G_j$.
 \end{enumerate}
 Then 
 \begin{enumerate}[(i)]\item
 $I$ does not depend on the choice of the polynomials  $F(i,j)$.
\item The cohomology ring \(H^*(Y(X_{\mathcal F}), \Q)\) is isomorphic to $B_{\mathcal F}[t_1,\ldots ,t_m]/I$.
\end{enumerate}
\end{theorem}

%\begin{remark} 1) Notice that if $G_j$ and $G_h$ are not comparable, then $t_jt_h\in I$. Indeed, let us take $A=\{j,h\}$ and let $i$ be  such that if $G_j\cap G_h=\emptyset$, $i=0$ or otherwise $G_i$ is a connected component of $G_j\cap G_h$.
%It both cases it follows that for the pair $(i,A)$:
% \[ F(i,A)=t_jt_h.\]
% 
% 2) In fact from this it follows that necessarily if the are  $j,h\in A$  such that $G_j$ and $G_h$ are not comparable, the relation $F(i,A)$, is a consequence of the relation $t_jt_h$ considered above. 
% 
% Then  we can assume that unless $A=\emptyset$, we can order $A=\{j_1,\ldots , j_r\}$ in such a way that $(G_i\subsetneq G_{j_1}\subsetneq\cdots\subsetneq G_{j_r})$. If this is the case $M=G_{j_1}$. Therefore the relations $F(i,A)$ are  a consequence of the relations  $$F(i,\emptyset)=P^X_{G_i}(\sum_{h\in B_i}-t_{h}), \ \ \text{or}\ \ \ \  F(i,j)=P^{G_{j}}_{G_i}(\sum_{h\in B_i}-t_{h})t_{j},$$
% with $G_i\subsetneq G_j$.\textcolor{blue}{ho messo relazioni al plurale perch\'e ho capito che volevi parlare in generale: se si fosse tenuta la descrizione della singola $F(i,A)$ si sarebbe dovuta scrivere $F(i,j_1)$}
% 
% \end{remark}

More generally one can compute the cohomology algebra of every stratum $Y_{(\underline G, C)}$ of $Y(X_{\mathcal F})$ as follows (in fact in Theorem 9.1 of \cite{DCG2} one of the relations, the relation (1),  was stated in a incorrect way;  this was corrected in Theorem 4.3 of \cite{mocipagaria2020}). 

First of all one shows that the restriction map

$$r_{(\underline G, C)}:H^*(Y(X_{\mathcal F}), \Q)\to H^*(Y_{(\underline G, C)}, \Q)$$
is surjective.

If a ray $c$ is such that $c\notin V_{\Gamma_{j_1}}$, 
\begin{equation}\label{prima3}r_{(\underline G, C)}(x_c)=0. \end{equation}

If $G_i$  is such that $C$ is not contained in $V_{\Gamma_i}$ or $\underline G\cup\{G_i\}$  
%\textcolor{red}{direi $\cup \{G_i\}$} 
cannot be reordered into a  flag we have 
\begin{equation} \label{prima6}r_{(\underline G, C)}(t_i)=0 .\end{equation}

%Similarly we have that if $\underline G\cup\{G_j\}$ cannot be reordered into a  flag \begin{equation}\label{prima3}r_{(\underline G, C)}(t_j)=0. \end{equation}

%Since we need the statements only in our special situation we will explain what we get. 

% Fix a pair  $(i,A)$ with   \(i\in \{1,...,m\}\), and    \(A\subset \{1,...,m\}\) such that if \(j\in A\) then \( G_i\subsetneq G_j\). Set $\mathcal S_i=\{h|G_{i_h}\in\underline G, G_{i_h}\supsetneq G_i\}$ and   \(B_i=\{h\: | \: G_h\subseteq G_i\}\).
%Denote by \(M_{\underline G}\) the unique connected component of  \(\bigcap_{j\in A\cup \mathcal S_i}G_j \) \textcolor{blue}{qui c'\`e un piccolo problema di notazione perch\'e se $j\in \mathcal S_i$ allora vuol dire che $G_{i_j}\in \mathcal S_i$, non $G_{j}\in \mathcal S_i$} that contains  $G_i$. We set:
%  \[ F_{\underline G}(i,A)=P^{M_{\underline G}}_{G_i}(\sum_{h\in B_i}-t_{h})\prod_{j\in A}t_{j}. \]
%  We know that  \( F_{\underline G}(i,A)\) is a relation in  \(H^*(Y_{(\underline G, C)}, \Q)\).
  
  Let us take a pair   $(i,j)$ with   \(i\in \{1,...,m\}\), and    \(j\in \{1,...,m+1\}\) in such a way that  \( G_i\subsetneq G_j\) and set $\mathcal S_i=\{s|G_{i_s}\in\underline G, G_{i_s}\supsetneq G_i\}$.

Let us start with a pair $(i,m+1 )$. 
 If  $\mathcal S_i=\emptyset$ one has the  relation $$F(i,m+1)=P^{G_{m+1}}_{G_i}(\sum_{h\in B_i}-t_{h}).$$
 which already holds in \(H^*(Y(X_{\mathcal F}))\).
 
 Otherwise, set $k=\min (s|\ s\in \mathcal S_i)$. 
One has the  relation  \begin{equation}\label{prima2}F_{\underline G}(i,m+1)=P^{G_{i_k}}_{G_i}(\sum_{h\in B_i}-t_{h})\end{equation}
in $H^*(Y_{(\underline G,C)},\mathbb Q)$.

\bigskip
Now let us consider the case of a pair $(i,j)$ with $j\leq m$. 
If  $\underline G\cup\{G_j\}$ cannot be reordered into a flag we already know from the relation \eqref{prima6} that $r_{(\underline G, C)}(t_j)=0$.
  
%Since by definition the relation  $F_{\underline G}(i,j)$ is divisible by $t_j$, this relation is a multiple of $F_{\underline G}(i',j)$.

Assume now that $\underline G\cup\{G_j\}$ can be reordered in a flag. Then also $\mathcal S_i\cup \{G_j\}$ is a flag and let $H$ be its smallest element. 

If $H=G_j$ and $G_j\notin \mathcal S_i$, we get the relation        
 $$  F(i,j)=P^{G_{j}}_{G_i}(\sum_{h\in B_i}-t_{h})t_j.$$
 which already holds in \(H^*(Y(X_{\mathcal F}))\).
 
 If $H=G_{i_k}\in \mathcal S_i$ we get the relation
  \( F_{\underline G}(i,j)=P^{G_{i_k}}_{G_i}(\sum_{h\in B_i}-t_{h})t_{j},\) which is a consequence of (\ref{prima2}).

\begin{theorem}
\label{teorelationsstrata} For any pair $(\underline G, C)$ with $C\subset V_{G_{i_1}}$,   the cohomology ring  \(H^*(Y_{(\underline G, C)}, \Q)\) is the quotient of the polynomial ring $B_{C,\mathcal F}[t_1,\ldots ,t_m]$  modulo the ideal generated by  
   \begin{enumerate}\item the image of the ideal $I$ modulo the quotient homomorphism
 $$\pi:B_{\mathcal F}[t_1,\ldots ,t_m]\to B_{C,\mathcal F}[t_1,\ldots ,t_m].$$
%\item the polynomials   $F(i,A)$ defined above.
% \end{enumerate}
% Then 
% \begin{enumerate}[(i)]\item
% $I$ does not depend on the choice of the polynomials  $F(i,A)$.
\item The relations (\ref{prima3}),  \eqref{prima6} and (\ref{prima2}).

%\item The polynomials \(F_{\underline G}(i,A)\), for every pair  $(i,A)$ with   \(i\in \{1,...,m\}\) and    \(A\subset \{1,...,m\}\) such that if \(j\in A\) then \( G_i\subsetneq G_j\), $|A|\leq 1$. 
\end{enumerate}
\end{theorem}
% \textcolor{red}{Vedi se mi sono scordato qualche relazione. Ho l'impressione che  \(H^*(Y_{(\underline G, C)}, \Q)=B_{\mathcal F}[t_1,\ldots ,t_m]/[I:\prod_{c\in C}x_c\prod_{G_j\in\underline G}t_j]\). Se questo \`e vero magari lo \`e anche nel caso di un building qualunque. }\textcolor{blue}{per ora ho verificato che le relazioni del teorema sono incluse in $[I:\prod_{c\in C}x_c\prod_{G_j\in\underline G}t_j]$}.
 
We can now apply this to give a presentation of the differential graded algebra associated to  \(Y(X_{\mathcal F})\) and the divisor with normal crossings  \(Y(X_{\mathcal F})\setminus \emme(\A)\) following 
\cite{Mor}. Recall that in our case this algebra is the direct sum
$$M_{\mathcal F}=\oplus_{(\underline G,C)}H^*(Y_{(\underline G, C)}, \Q)[-n_{(\underline G, C)}]$$
with $n_{(\underline G, C)}$ equal to  $\dim C+|\underline G|$ which is the codimension of $Y_{(\underline G, C)}$.

%In order to do so, we start considering the  differential graded algebra $\mathcal C_{\underline F}$. 
In order to do so, we  take the algebra  $\mathcal B=\mathbb Q[t_1,\ldots , t_m]\otimes \bigwedge (\kappa_1,\ldots ,\kappa_m)/ K$ where $K$ is the ideal generated by the products $t_it_j$, $t_i\kappa_j$, $\kappa_i\kappa_j$ whenever $G_i$ and $G_j$ are not comparable. 

We grade $\mathcal B$ by setting
 $\deg t_j=2$ and $\deg \kappa_j=1$ and we remark that  the usual differential on $\mathcal B=\mathbb Q[t_1,\ldots , t_m]\otimes \bigwedge (\kappa_1,\ldots ,\kappa_m)$ given by $d(\kappa_j)=t_j$ preserves  $K$  so that $\mathcal B$ inherits a degree 1 differential $d_{\mathcal B}$.

We can then consider the algebra $\mathcal C_{\mathcal F}\otimes \mathcal B$ with differential $d_{\mathcal C_{\mathcal F}}\otimes 1+1\otimes  d_{\mathcal B}$. 

Remark that $B_{\mathcal F}$ is the subalgebra of $\mathcal C_{\mathcal F}$ consisting of element of bidegree $(2n,0)$, $n\geq 0$ and so the polynomial ring $B_{\mathcal F}[t_1,\ldots t_m]$ is a subalgebra of $\mathcal C_{\mathcal F}\otimes \mathcal B$.  Using this remark we  can take   the ideal $\Theta_{\mathcal F}$  in  
  $\mathcal C_{\mathcal F}\otimes \mathcal B$ generated by the elements
\begin{enumerate}
\item $x_ct_j$,  $\tau_ct_j$, $x_c\kappa_j$, $\tau_c\kappa_j$, $c\notin V_{G_{j}}$. 
%\textcolor{blue}{qui e nel seguito i $t_j$ sarebbero invece gli $y_j$}
\item
 $F(i,j)$, for $G_j\supsetneq G_i$, $i=1,\ldots m$, $j=1,\ldots ,m+1$.
\item  $P^{G_{j}}_{G_i}(\sum_{h\in B_i}-t_{h}) \kappa_j,$ with   \(i\in \{1,...,m\}\), and    \(j\in \{1,...,m\}\) in such a way that  \( G_i\subsetneq G_j\).
\end{enumerate}
Observe that $\Theta_{\mathcal F}$  is preserved by the differential $d_{\mathcal C_{\mathcal F}}\otimes 1+1\otimes  d_{\mathcal B}$. It follows that we get an induced  differential $d_{N_{\mathcal F}}$ on the algebra $N_{\mathcal F}=\mathcal C_{\mathcal F}\otimes \mathcal B/\Theta_{\mathcal F}$. 

We know that $\Theta_{\mathcal F}$ is a graded ideal so that $N_{\mathcal F}$ is also graded and the differential $d_{ N_{\mathcal F}}$ is of degree 1.
\begin{theorem} The differential graded algebra $( N_{\mathcal F},d_{N_{\mathcal F}})$ is isomorphic to the Morgan algebra $(M_{\mathcal F}, d_{M_{\mathcal F}})$.\end{theorem}
\begin{proof}  The proof given in \cite{mocipagaria2020} can be applied verbatim in this more special case.

 \end{proof}
\section{A limit and the rational homotopy type of $M(\mathcal A)$}
\label{sec:limit}

Before we start, let us briefly discuss the generators of $\Theta_{\mathcal F}$. Remark that $c\notin V_{\Gamma_{j}}$ if and only if the function $s_c=\rho_{\mathcal F} (x_c)$ vanishes on $ V_{\Gamma_{j}}$. Indeed the support of $s_c$ is the interior of $S(c)$ the star of $c$ and such interior intersects $V_{\Gamma_{j}}$ if and only if $c\in V_{\Gamma_{j}}$.

Thus the first relations can be written as 
\begin{equation}\label{lefirst}xt_j,  \ \tau  t_j,\ x \kappa_j, \ \tau\kappa_j,\end{equation} for $x,\tau \in \{f\in S_{\mathcal F}| \ \rho_{\mathcal F} (f)\equiv 0$ on $V_{\Gamma_{j}}\}$.
%\textcolor{blue}{anche qui i  $t_j$ sarebbero invece gli $y_j$}

As for the elements 
$$\chi_{\mathcal F}^-=\sum_{c\ \text{ray}}\min(0,\langle \chi,c\rangle) x_c,$$
appearing in the definition of
 $P^{G_j}_{G_i}(t):=\prod_{\chi\in \underline \chi\setminus (\underline \chi\cap \Gamma_j)}(t-\chi_{\mathcal F}^-)$, 
% \textcolor{red}{sempre dubbi sul $\Gamma'$}
 we remark that $\rho_{\mathcal F}(\chi_{\mathcal F}^-)=\chi^-$ where 
 $$\chi^-(v)=\min (0, \chi(v))$$
% \textcolor{blue}{ mi pare che ci voglia  $\chi^-(v)=\min (0, \chi(v))$}
 
 for each $v\in V$.
 
 Notice that if $\mathcal G$ is a refinement of $\mathcal F$ an equal sign linear function relative to $\mathcal F$ is also equal sign relative to $\mathcal G$,
  so if $\mathcal F$ is compatible with \(\A\) also $\mathcal G$ is compatible with \(\A\).
 
 The first consequence of this fact is
 
 \begin{proposition}\label{preimmagine} Let $\mathcal F$ be a fan compatible with \(\A\). Let $\mathcal G$ be a refinement of $\mathcal F$ and
  $\psi_\mathcal G^\mathcal F:X_\mathcal G\to X_\mathcal F$   the   unique $T$-equivariant projective morphism  extending the identity on $T$. 

Then,  for each  layer  $K_{\Gamma,\phi}$ of $\mathcal A$, the preimage of the closure of $K_{\Gamma,\phi}$ in $X_\mathcal F$ is 
 the closure of $K_{\Gamma,\phi}$ in $X_\mathcal G.$
  \end{proposition}
\begin{proof} Fix the layer $K_{\Gamma,\phi}$ of $\mathcal A$. Clearly we can identify $K_{\Gamma,\phi}$ with the torus $T_{\Gamma}=\cap_{\chi\in\Gamma}\ker e^{2\pi\imath\chi} $ and the restriction $\psi_\mathcal G^\mathcal F$
to $K_{\Gamma,\phi}$ is the identity. 

We know by the compatibility of   $\mathcal F$  and $\mathcal G$   with \(\A\), that a cone $C$ in  $\mathcal F$ or in
$\mathcal G$
has either relative interior disjoint from $V_\Gamma$ or it is contained in $V_\Gamma$.

The cones contained in $V_\Gamma$ define   smooth  projective fans $\mathcal F_\Gamma $ and $\mathcal G_\Gamma$ respectively and $\mathcal G_\Gamma$ is a refinement of  $\mathcal F_\Gamma $.

We can then identify  the closure of  $Z_\mathcal F$ (resp. $Z_\mathcal G$) of $K_{\Gamma,\phi}$ in $X_\mathcal F$ (resp. $X_\mathcal G$ ) with the $T_\Gamma$ variety associated to the fan $\mathcal F_\Gamma $ (resp. $\mathcal G_\Gamma $) and the restriction of  $\psi_\mathcal G^\mathcal F$ to  $Z_\mathcal G$ with the unique $T_\Gamma$ equivariant morphism extending the identity on  $K_{\Gamma,\phi}$ (identified with $T_\Gamma$).

We need to prove that  $(\psi_\mathcal G^\mathcal F)^{-1}(Z_\mathcal F)=Z_\mathcal G$. In order to see this let us take a cone $C\in\mathcal F$. Consider the new $T$-orbit $\mathcal O_C\subset  X_\mathcal F$ corresponding to $C$. We know that  $(\psi_\mathcal G^\mathcal F)^{-1}(\mathcal O_C)$ is the union of the $T$-orbits $\mathcal O_{C'}\subset  X_\mathcal G$ corresponding to the cones  $C'\in \mathcal G$ whose relative interior is contained in the relative interior of $ C$. 

Also  as a $\mathcal O_C\simeq T/T_{\Gamma_C}$, were $$T_{\Gamma_C}=\cap_{\chi\in X^*( T),\langle \chi , C\rangle =0}\ker e^{2\pi\imath\chi},$$
the restriction of the map $\psi_\mathcal G^\mathcal F$ to a $T$-orbit $\mathcal O_{C'}\subset  X_\mathcal G$ corresponding to a cone $C'\in \mathcal G$ whose relative interior is contained in the relative
 interior of $ C$ can be then identified with the projection $T/T_{\Gamma_{C'}}\to T/T_{\Gamma_C}$ whose fiber is the torus $T_{\Gamma_C}/T_{\Gamma_{C'}}$.
 
 Having recalled these facts let us examine the preimage of $Z_\mathcal F$. Fix a cone $C\in\mathcal F$.
 
  If $C$ is not contained in $V_\Gamma$ then $\mathcal O_C\cap Z_\mathcal F=\emptyset$, so no orbit $\mathcal O_{C'}\subset  X_\mathcal G$ corresponding to a cone $C'\in \mathcal G$ whose relative interior is contained in the relative interior of $ C$ intersects $(\psi_\mathcal G^\mathcal F)^{-1}(Z_\mathcal F)$.
  
  If $C$ is  contained in $V_\Gamma$  then also every cone $C'\in \mathcal G$ whose relative interior is contained in the relative interior of $ C$ is contained in $V_\Gamma$. Furthermore we have inclusions
 $ T_{\Gamma_{C'}}\subset T_{\Gamma_{C}}\subset T_{\Gamma}$. From this and the description of the restriction of the map $\psi_\mathcal G^\mathcal F$ to  $\mathcal O_{C'}$, it follows  that 
 $(\psi_\mathcal G^\mathcal F)^{-1}(Z_\mathcal F)\cap \mathcal O_{C'}\subset Z_\mathcal G$, proving our claim.
\end{proof}

 At this point we observe that the universal property of Blowing up (see, \cite{Hart}, pp.164-165) implies that we get a morphism 
 $$\nu_\mathcal G^\mathcal F:Y(X_{\mathcal G})\to Y(X_{\mathcal F})$$ extending the identity of $\emme(\A)$.
 
The map  $\nu_\mathcal G^\mathcal F$ then induces a homomorphism \[\Phi_\mathcal G^\mathcal F:N_\mathcal F\to N_\mathcal G\]
 of differential graded algebras having the following properties:
\(\Phi_\mathcal G^\mathcal F\) coincides with \(\chi_\mathcal G^\mathcal F\) on $\mathcal C_\mathcal F$ and one has $\Phi_\mathcal G^\mathcal F(t_i)=t_i$, $\Phi_\mathcal G^\mathcal F(\kappa_i)=\kappa_i$ for every $i=1,\ldots, m$. 

%Notice that $\Phi_\mathcal G^\mathcal F$ is induced by a natural morphism $\nu_\mathcal G^\mathcal F:Y(X_{\mathcal G})\to Y(X_{\mathcal F})$. 
%
%%\textcolor{blue}{provo ad aggiungere una  descrizione di questa mappa qui sotto:}
%\begin{remark}
%Since $\mathcal G$ is a refinement of $\mathcal F$, we have the morphism $\psi_\mathcal G^\mathcal F:X_\mathcal G\to X_\mathcal F$  which  is the   unique $T$-equivariant projective morphism  extending the identity on $T$. Thus $\psi_\mathcal G^\mathcal F$ is the identity on each of the layers  $K_{\Gamma,\phi}$ of $\mathcal A$  and induces a projective morphism on their closures which is equivariant with respect to the action of the torus $T_{\Gamma}=\cap_{\gamma\in\Gamma}\ker \gamma .$
%
%
%
%
% \textcolor{blue}{era stato gi\`a detto chi era $T_\Gamma$?.} 
%
%By uniqueness this  is the unique $T_\Gamma$-equivariant morphism associated to the two fans given by the cones of $\mathcal F$ (resp. $\mathcal G)$ contained in $V_\Gamma$.
%
%Since  
%$Y(X_\mathcal F)$ (resp. $Y(X_\mathcal G$)) are obtained by blowing up the  the closures of these layers as described in Section \ref{sec:toric}, the existence of $\nu_\mathcal G^\mathcal F$ follows.
%
%
%\end{remark}
%
%\textcolor{blue}{vedi se va meglio}

%\textcolor{blue}{fine aggiunta}
 
 Let us now remark that  the function $\chi^-$ depends only on $\chi$ and not on the choice of any particular fan. This allows us to define the polynomial 
  $P^{G_j}_{G_i}(t):=\prod_{\chi\in \underline \chi\setminus (\underline \chi\cap \Gamma_j)}(t-\chi^-)\in \mathcal C[t]$
%\textcolor{red}{sempre dubbi sul $\Gamma'$}
%  Thus we can now define  the differential graded algebra $N$ as the quotient of the algebra $\mathcal C\otimes \mathcal B$ modulo the ideal 
 We can then consider the algebra $\mathcal C\otimes \mathcal B$ with differential $d_{\mathcal C}\otimes 1+1\otimes  d_{\mathcal B}$. 

For any $\Gamma_j$ we consider the subspace $S_{\Gamma_j}$ which is the image modulo $V^*$ of the space of functions whose restriction to  $V_{\Gamma_{j}}$ is linear.

Now remark that the space $\Sigma$  surjects on  both $\mathcal C_{(2,0)}$ and  $C_{(0,1)}$. So, given $\upsilon \in \Sigma$, we may take the corresponding elements $s_\upsilon\in \mathcal C_{(2,0)}$ and 
$\sigma_\upsilon\in \mathcal C_{(0,1)}$.

In $\mathcal C\otimes \mathcal B$ we then take    the ideal $\Theta $ generated by the elements 
\begin{enumerate}
\item $s_\upsilon\ t_j$,  $ \sigma_\upsilon t_j$, $s_\upsilon\kappa_j$, $\sigma_\upsilon \kappa_j$, for $\upsilon \in S_{\Gamma_j}.$ 
%\textcolor{blue}{anche qui i  $t_j$ sarebbero invece gli $y_j$, inoltre S non mi pare definito. 

%\textcolor{blue}{S non mi pare definito.  Si tratta di $\Sigma$ quozientato le lineari?}
\item
 $F(i,j)$, for $G_j\supsetneq G_i$,  $i=1,\ldots m$, $j=1,\ldots ,m+1$. 
% \textcolor{red}{aggiunto anche qui dove variano i e j}
\item  $P^{G_{j}}_{G_i}(\sum_{h\in B_i}-t_{h}) \kappa_j,$ with   \(i\in \{1,...,m\}\), and    \(j\in \{1,...,m\}\) in such a way that  \( G_i\subsetneq G_j\).
\end{enumerate}
and set $N=\mathcal C\otimes \mathcal B/\Theta$.

Observe that $\Theta$  is preserved by the differential $d_{\mathcal C}\otimes 1+1\otimes  d_{\mathcal B}$. It follows that we get an induced  differential $d_{ N}$ on  $ N$. 

We know that $\Theta$ is a graded ideal so that $\mathcal N$ is also graded and the differential $d_{\mathcal N}$ is of degree 1.

From \eqref{defini} and the above observations we then deduce

\begin{theorem} 
\label{teo:limite}
$(N,d_{N})=\varinjlim_{\mathcal F}(N_{\mathcal F},d_{N_{\mathcal F}}).$\end{theorem}

We now want to remark a few facts about the algebras \((N_{\mathcal F},d_{N_{\mathcal F}})\) and \((N,d_{N})\). First of all we have seen that although the set of generators of the defining ideal of  \(N_{\mathcal F}\) depends on the choice of equal sign bases, by Theorem \ref{propindependent}(i) the algebra itself is independent on the choice  of these equal sign bases. 

%Our second obvious remark  is that if a vector has constant sign on a given cone $C$ it also has equal sign on eache of the cones of any refinemet of $C$. It follows that a choice of equal sign bases for a fan $\mathcal F$ is also a choice of equal sign bases for any refinement $\mathcal G$ of $\mathcal F$. 

Assume now that we have given two smooth projective fans $\mathcal F$ and $\mathcal F'$ each equipped with a choice of  equal sign bases for our arrangement. We know that we can find a common refinement $\mathcal G$ of  $\mathcal F$ and $\mathcal F'$. We have already remarked that   the two sets of    equal sign bases are both choices of  equal sign bases for $\mathcal G$, so that the algebra \((N_{\mathcal G},d_{N_{\mathcal G}})\) is independent from these choices.

We deduce
\begin{proposition}\label{independence} The differential graded algebra \((N,d_{N})\) is independent from any choice of equal sign bases.\end{proposition}

Let us now recall that according to \cite{Mor}, for smooth projective fans $\mathcal F$ as above the minimal model of the differential graded algebra \((N_{\mathcal F},d_{N_{\mathcal F}})\)  is isomorphic to the minimal model of \(\emme(\A)\) so it determines the rational homotopy type of \(\emme(\A)\) and in particular its cohomology is the cohomology ring $H^*(\emme(\A),\mathbb Q)\). Furthermore as we have already seen   the homomorphism  \(\Phi_\mathcal G^\mathcal F\) is induced by the morphism $\nu_\mathcal G^\mathcal F:Y(X_{\mathcal G})\to Y(X_{\mathcal F})\)  which is the identity when restricted to \(\emme(\A)\), so \(\Phi_\mathcal G^\mathcal F\)  is a quasi isomorphism of differential graded algebras. 

Since. taking homology of a chain complex is an exact functor we deduce that also the cohomology of the differential graded algebra \((N,d_{N})\) is $H^*(\emme(\A),\mathbb Q)\) and that the natural maps \((N_{\mathcal F},d_{N_{\mathcal F}})\to (N,d_{N})\) are quasi isomorphisms.

We deduce 
\begin{proposition}\label{quasifinale} The minimal model  of the differential graded algebra \((N,d_N)\)  is isomorphic to the minimal model of \(\emme(\A)\).\end{proposition}

To state our final result in its more general form, let us now consider an arbitrary toric arrangement $\mathcal A$ (i.e. we drop the assumption \(\A=\tilde \A\)). We recall that, according to the definition in the  Introduction,  the  {\em combinatorial data}  of  $\mathcal A$ are provided by the  following two sets:
\begin{enumerate} \item The partially ordered set $\tilde {\mathcal A}$  ordered by reverse inclusion.
\item The set of lattices $\Gamma$ for $\mathcal K_{\Gamma,\phi}\in \tilde {\mathcal A}$.
\end{enumerate} 
We have: 
 \begin{theorem}\label{finale} The rational homotopy type of the complement   \(\emme(\A)\) depends only on the combinatorial data of \(\A\).\end{theorem}
 \begin{proof} By Proposition \ref{quasifinale} the rational homotopy type of the complement   \(\emme(\A)\)  depends only by the differential graded algebra \((N,d_N)\). In turn we have that this algebra is defined only in terms of the combinatorial data of $\mathcal A$. Hence our claim follows.\end{proof}

\addcontentsline{toc}{section}{References}
%\nocite{*}
\bibliographystyle{acm}
\bibliography{Bibliogpre} 
\end{document}